\documentclass[11pt]{amsart}
\usepackage[margin=1in]{geometry}
\usepackage{amssymb,amsmath,latexsym,mathtools}
\usepackage{mathbbol}
\usepackage{enumerate}
\usepackage{graphicx}            
\usepackage{amsfonts}           
\usepackage{amsthm}               
\usepackage{url}
\usepackage{marvosym}
\usepackage{color}
\usepackage[colorlinks,pdfencoding=auto]{hyperref}
\usepackage{soul}
\usepackage{bm}
\usepackage{tikz-cd}
\usepackage[inline]{enumitem}
\usepackage{tasks}
\usepackage{stmaryrd}
\usepackage{thm-restate}

\DeclareSymbolFontAlphabet{\mathbb}{AMSb}
\DeclareSymbolFontAlphabet{\mathbbl}{bbold}
\DeclareMathAlphabet\mathbfcal{OMS}{cmsy}{b}{n}

\newtheorem{theorem}{Theorem}[section]
\newtheorem*{theorem*}{Theorem}
\newtheorem{lemma}[theorem]{Lemma}
\newtheorem{proposition}[theorem]{Proposition}
\newtheorem{corollary}[theorem]{Corollary}

\theoremstyle{definition}
\newtheorem{definition}[theorem]{Definition}
\newtheorem{example}[theorem]{Example}
\newtheorem{remark}[theorem]{Remark}

\newcommand{\op}[1]{\operatorname{#1}}
\newcommand{\defn}[1]{{\color{blue} \it {#1}}}

\newcommand{\CC}{\mathbb{C}}

\begin{document}

\title[Counting chains via the $W$-Laplacian]{Counting chains in the noncrossing partition lattice via the $W$-Laplacian}
\author{ Guillaume Chapuy,\ Theo Douvropoulos}
\thanks{This project (both authors) has received funding from the European Research
  Council (ERC) under the European Union’s Horizon 2020 research and
  innovation programme (grant agreement No. ERC-2016-STG 716083
  “CombiTop”).}

\newcommand{\Address}{{
  \bigskip
  \footnotesize
	
  Theo~Douvropoulos, \textsc{IRIF, UMR CNRS 8243, Universit\'e Paris Diderot, Paris 7, France}\par\nopagebreak
  \textit{E-mail address:} \texttt{douvr001@irif.fr}	

}}

\begin{abstract}
We give an elementary, case-free, Coxeter-theoretic derivation of the formula $h^nn!/|W|$ for the number of maximal chains in the noncrossing partition lattice $NC(W)$ of a real reflection group $W$. Our proof proceeds by comparing the Deligne-Reading recursion with a parabolic recursion for the characteristic polynomial of the $W$-Laplacian matrix considered in our previous work. We further discuss the consequences of this formula for the geometric group theory of spherical and affine Artin groups.
\end{abstract}

\date{}
\maketitle

\section{Introduction}


The lattice of noncrossing partitions $NC(W)$ of a reflection group $W$ has been a prominent figure in the lands of algebraic combinatorics and geometric group theory -- it holds dual citizenship! The general saga of Coxeter combinatorics involves imagining (and sometimes proving) generalizations of theorems that are known for the symmetric group $S_n$ to all reflection groups $W$. If moreover the old theorems involve 
the Catalan numbers or one of their variants, we are sure to have found ourselves working on Coxeter-Catalan combinatorics.

In the early 90s, Reiner \cite{reiner_classical} introduced pictorial generalizations of Kreweras' original noncrossing partitions (famously enumerated by the Catalan numbers \cite{krew}) associating them with the classical reflection groups $B_n$ and $D_n$. In the next few years, after many works including \cite{biane,brady_watt,bessis_dual} it became clear to the community that noncrossing partitions can be defined for all reflection groups and in a way that is both natural and useful (see also the surveys \cite{simion,mccammond}). The noncrossing lattice $NC(W)$ was defined then as the interval $[1,c]_{\leq_{\mathcal{R}}}$ in the group $W$ below a Coxeter element $c\in W$ under the absolute order $\leq_{\mathcal{R}}$.

One of the most attractive features of this definition was the remarkable numerological properties that $NC(W)$ exhibited. A main invariant of the lattice, its Zeta polynomial $\mathcal{Z}(NC(W),k)$ which counts chains of length $k$ in $NC(W)$, could be expressed with a single formula (valid for all reflection groups) involving only certain invariant theoretic constants of $W$. A special case of that is an even simpler formula for the number $MC(W)$ of \emph{maximal, saturated} chains in $NC(W)$. The following theorem is due to Chapoton \cite{chapoton} who combined results from the works  \cite{reiner_classical,ath_reiner}.

\begin{theorem}[{Chapoton's formula and chain number, \cite[Prop.~9]{chapoton}}]\label{Thm intro chapoton}\ \newline 
For an irreducible real reflection group $W$, of rank $n$ with Coxeter number $h$ and fundamental degrees $(d_i)_{i=1}^n$, the Zeta polynomial and number of maximal chains of $NC(W)$ are given by:
\[
i)\quad \mathcal{Z}(NC(W),k)=\prod_{i=1}^n\dfrac{kh+d_i}{d_i}\qquad\qquad\text{and}\qquad ii)\qquad MC(W)=\dfrac{h^nn!}{|W|}.
\]
\end{theorem}

The second formula of Thm.~\ref{Thm intro chapoton} has a fascinating history. The number $MC(W)$ counts equivalently the minimum length reflection factorizations of a Coxeter element $c$, and with that interpretation it was conjectured by Looijenga \cite{Looijenga} to agree with the degree of a finite, weighted-homogeneous morphism in singularity theory, now known as the Lyashko-Looijenga ($LL$) map. It was then Arnold who first \cite[Thm.~11]{arnold_ICM} noticed the uniform expression $h^nn!/|W|$ for the degree of this $LL$ map.

Now, Looijenga's conjecture was settled almost immediately by Deligne (crediting Tits and Zagier) \cite{Deligne,kluitmann-thesis} who calculated the quantities $MC(W)$ and compared them with the already known degrees of the $LL$ maps. Deligne gave a uniform recursion on the set of factorizations, which led to a recurrence on the cardinalities $MC(W)$ (in terms of some parabolic subgroups of $W$). This same approach was rediscovered later by Reading \cite{reading_chains}, but both works solved this recurrence in a case-by-case way. The main contribution of our paper is a uniform derivation of the chain number formula of Thm.~\ref{Thm intro chapoton}:~(ii) \emph{via a case-free solution of the Deligne-Reading recurrence}. This is the second uniform but first \emph{Coxeter-theoretic} such proof; see Remark~\ref{Rem: Advantages of our proof}. The more general part~(i) of Thm.~\ref{Thm intro chapoton} seems to still be out of reach, although we discuss an approach in Sec.~\ref{Sec: Fomin-Reading}.

\subsection*{A case-free solution of the Deligne-Reading recurrence}

This argument became possible and in fact relatively easy after our recent work \cite{chapuy_douvr}; we briefly sketch the main points here. The Deligne-Reading recursion on the set $MC(W)$ is given below; if $S$ is the set of simple generators and $h$ the Coxeter number of $W$ then
\begin{equation}
MC(W)=\dfrac{h}{2}\cdot \sum_{s\in S}MC(W_{\langle s\rangle}),\label{eq intro DR recursion}
\end{equation}
where  $W_{\langle s\rangle}$ is the (standard) parabolic subgroup of $W$ generated by the simple reflections $S\setminus \{s\}$. To prove that the formula of Thm.~\ref{Thm intro chapoton}:~(ii) satisfies the corresponding recurrence, we need a seemingly complicated equality among the Coxeter numbers of $W$ and those of its parabolic subgroups (which may be reducible).  If we write $\{h_i(W)\}_{i=1}^n$ for the \emph{multiset of Coxeter numbers} of $W$, and translate the previous recursion from standard to arbitrary parabolic subgroups of $W$, we need to show that
\begin{equation}
h^{n-1}\cdot n=\sum_{L\in\mathcal{L}^1_{\mathcal{A}_W}}\prod_{i=1}^{n-1}h_i(W_L),\label{eq intro h^(n-1)-h_i(W_L)}
\end{equation}
where the sum is over $1$-dimensional flats (lines $L$) of the intersection lattice of $W$. As it turns out this is a special case of a more general theorem relating Coxeter numbers, which we give below. It is a corollary of the spectral study of the \defn{$W$-Laplacian}, a $(n\times n)$ matrix associated with $W$ (see Sec.~\ref{Sec: W-Laplacian}). To recover \eqref{eq intro h^(n-1)-h_i(W_L)}, one need only compare the coefficients of $t^1$ in the two sides of Thm~\ref{Thm: rec cox nums}.

\begin{restatable*}[{Parabolic recursion of Coxeter numbers, \cite{chapuy_douvr}}]{theorem}{hrecursion}\label{Thm: rec cox nums}
For an irreducible, real reflection group $W$ with reflection arrangement $\mathcal{A}$ and Coxeter number $h$, we have that 
\[
(t+h)^n=\sum_{X\in\mathcal{L}_{\mathcal{A}}}\prod_{i=1}^{\op{codim}(X)}h_i(W_X)\cdot t^{\op{dim}(X)},
\]
where $\{h_i(W_X)\}$ denotes the \emph{multiset} of Coxeter numbers of $W_X$.
\end{restatable*}

We try to keep the paper self sufficient, and in doing so we develop in some detail in Sec.~\ref{Sec: prelim} the various objects that appear in the discussion above. A reader familiar with Coxeter groups could directly read the $W$-Laplacian construction in Sec.~\ref{Sec: W-Laplacian} and then the quick proof in Sec.~\ref{Sec: The proof}. In Sec.~\ref{Sec: Fomin-Reading} we briefly discuss a possible approach towards the complete proof of Chapoton's formula.

\subsection*{Implications on the geometric group theory of $W$}
One of the most significant and less well known applications of the formula for $MC(W)$ is in the proof of the concordance between the standard and dual braid presentations of spherical Artin groups. The equivalence of the two presentations has had tremendous influence in recent years but no uniform proof for it was known. After Bessis' work \cite{Bessis}, the last ingredient needed that still relied on case-by-case calculations was in fact just the formula for the chain number $MC(W)$; we demonstrate this in Section~\ref{Sec: implications braid group}.

\section{Reflection groups and the noncrossing partition lattice}
\label{Sec: prelim}

In this section we review some of the basic theory of real reflection groups, including the Coxeter complex geometry, and we introduce the noncrossing partition lattice. We point the reader to the standard references \cite{kane-book-reflection-groups,HUM} for anything we might leave unexplained.

A {\color{blue} \emph{(real) reflection group} $W$} of rank $n$ is a finite subgroup of $\op{GL}(\mathbb{R}^n)$ generated by Euclidean reflections. We will say that $W$ is {\color{blue}\emph{irreducible}} if there is no non-trivial, proper $W$-invariant linear subspace of $\mathbb{R}^n$. The collection of fixed hyperplanes of \emph{all} reflections of $W$ is called the {\color{blue}\emph{reflection arrangement}} $\mathcal{A}_W$ (or simply $\mathcal{A}$ in what follows) and its geometry was shown by Coxeter \cite{coxeter-annals} to determine much of the theory of $W$. In particular, the hyperplanes of $\mathcal{A}$ divide the complement $\mathbb{R}^n\setminus \mathcal{A}$ into simplicial open chambers whose closures form fundamental domains for the action of $W$. Given a choice of such a {\color{blue}\emph{fundamental chamber $\mathcal{C}$}}, the reflections across its facets determine a set $S$ of {\color{blue} \emph{simple generators}} for $W$ and the angles between them lead to the Coxeter presentation of $W$ (see \S~\ref{Sec: Artin-presentation}).

\subsection{The Coxeter complex and parabolic subgroups}

Given such a set $S:=\{s_i\}_{i=1}^n$ of simple generators for $W$, the groups $W_I:=\langle I\rangle$ generated by any subset $I\subset S$ are called {\color{blue}\emph{standard parabolic subgroups}} of $W$ and are also (real) reflection groups. They correspond bijectively to the faces of the fundamental chamber $\mathcal{C}$ being in fact, after a theorem of Steinberg, their pointwise stabilizers. More generally, Steinberg's theorem (see \cite[\S~4.2.3]{broue-book-braid-groups}) states that the pointwise stabilizer of \emph{any} subset of $\mathbb{R}^n$  fixes a whole linear subspace $X$ that is the intersection of some collection of hyperplanes of $W$. Such spaces $X$ are called {\color{blue}\emph{flats}} of the reflection arrangement $\mathcal{A}$ of $W$ and they form the {\color{blue}\emph{intersection lattice $\mathcal{L}_{\mathcal{A}}$}}; they are in bijection with their stabilizers $W_X$ which are simply called {\color{blue} parabolic subgroups}. We will denote $W$-orbits of flats via $[X]\in\mathcal{L}_{\mathcal{A}}/W$; they correspond to orbits of parabolic subgroups under conjugation.

The intersection of the reflection arrangement $\mathcal{A}$ with the unit sphere $\mathbb{S}^{n-1}$ in $\mathbb{R}^n$ determines a triangulation of $\mathbb{S}^{n-1}$ which is known as the {\color{blue} \emph{Coxeter complex}} of $W$, denoted $C(W)$. It has $|W|$-many facets which correspond to the chambers of $\mathcal{A}$, while its faces of smaller dimension correspond to chambers of the \emph{restricted arrangements} $\mathcal{A}^X:=\{H\in\mathcal{A}\ |\ H\supset X\}$ for any flat $X\in\mathcal{L}$. The (pointwise) stabilizer of any face $F$ of the Coxeter complex $C(W)$ is then the parabolic subgroup $W_X$, where $X$ is the linear span of $F$ in $\mathbb{R}^n$. Since $W$ acts transitively on the facets of the complex $C(W)$, any face is in the same $W$-orbit as a face of the fundamental chamber $\mathcal{C}$. This means that the collection of parabolic subgroups $W_X$ is the complete family of subgroups of $W$ that are conjugate to some standard parabolic $W_I$ for any subset $I\subset S$.

It turns out however that the standard parabolic subgroups $W_I$ do not form a set of representatives; that is, we might have $W_I\cong W_J$ for different subsets $I$ and $J$ of $S$. It will be important in what follows to have a finer understanding of this situation. We write $\nu^{}_X$ for the \emph{number of standard parabolic subgroups conjugate to $W_X$} (clearly this is really indexed by the $W$-orbit $[X]$ but we avoid the heavier notation $\nu^{}_{[X]}$). Following Orlik and Solomon \cite[(4.2)]{OS_nu} it is easy to give a formula for $\nu^{}_X$ via a simple double counting argument. We want to enumerate all pairs $(F,C)$ where $F\subset C$ is a face of the facet $C$ of the Coxeter complex $C(W)$, such that $W_F\cong W_X$. Since all chambers are $W$-conjugate, it is clear that there are $\nu^{}_X\cdot |W|$ such pairs. On the other hand, if $c(\mathcal{A}^X)$ denotes the number of chambers in the restricted arrangement $\mathcal{A}^X$, and $N(X)$ the \emph{setwise} normalizer of the flat $X$, there is a total of $[W:N(X)]\cdot c(\mathcal{A}^X)$ faces with $W$-stabilizer conjugate to $W_X$. Each of those faces is incident to $|W_X|$-many chambers so that we have
$$ \nu^{}_X\cdot |W|=\big([W:N(X)]\cdot c(\mathcal{A}^X)\big)\cdot |W_X|,$$ by the double counting argument, and this in turn implies the formula \begin{equation}
\nu^{}_X=\dfrac{c(\mathcal{A}^X)}{[N(X):W_X]}.\label{EQ: OS formula nu_X}
\end{equation}

We are now ready to prove a short technical lemma that will help us translate the Deligne-Reading recursion \eqref{eq intro DR recursion} from one on the simple generators of $W$ to a recursion on flats of dimension $1$ (lines) of the intersection lattice $\mathcal{L}_{\mathcal{A}}$. Recall the notation from the introduction where $W_{\langle s\rangle}$ is defined as the standard parabolic subgroup $W_J$ with $J=S\setminus\{s\}$. We further write $\mathcal{L}^k_{\mathcal{A}}$ for the collection of flats of the arrangement $\mathcal{A}$ that have dimension equal to $k$.

\begin{lemma}[From recursions on $S$ to recursions on flats]\label{Lem: simples to flats}\ 
Let $W$ be a finite Coxeter group with reflection arrangement $\mathcal{A}$, set of simple generators $S$, and lattice of flats $\mathcal{L}_{\mathcal{A}}$. If $F$ is any function on Coxeter types\footnote{That is, a function on (finite) Coxeter groups $W$ that only depends on the isomorphism type of $W$.} defined at least for $W$ and its maximal parabolic subgroups $W_{\langle s\rangle}$ and if $a\in\mathbb{R}$ is an arbitrary parameter, then the following two equations are equivalent:
$$F(W)=a\cdot \sum_{s\in S}F(W_{\langle s\rangle})\quad \iff \quad |W|\cdot F(W)=2a\cdot \sum_{L\in\mathcal{L}_{\mathcal{A}}^1}|W_L|\cdot F(W_L).$$
\end{lemma}

\begin{proof}
We can start with the first recursion and rewrite its right hand side as 
$$a\cdot \sum_{s\in S}F(W_{\langle s\rangle})=a\cdot\!\!\!\!\!\!\! \sum_{[L]\in\mathcal{L}^1_{\mathcal{A}}/W}\nu_{L}\cdot F(W_L),$$ where the summation is over $W$-orbits of 1-dimensional flats and where $\nu^{}_L$, as earlier, counts the number of standard parabolic subgroups conjugate to $W_L$. To apply formula \eqref{EQ: OS formula nu_X} notice that for an 1-dimensional flat $L$, we always have $c(\mathcal{A}^L)=2$, so that the right hand side becomes 
$$a\cdot\!\!\!\!\!\!\! \sum_{[L]\in\mathcal{L}^1_W/W}\nu_{L}\cdot F(W_L)=2a\cdot\!\!\!\!\!\!\! \sum_{[L]\in\mathcal{L}^1_W/W}\dfrac{F(W_L)}{[N(L):W_L]}=2a\cdot\!\!\sum_{L\in\mathcal{L}^1_{\mathcal{A}}}\dfrac{1}{[W:N(L)]}\cdot \dfrac{F(W_L)}{[N(L):W_L]},$$ with the second equality replacing the summation over \emph{orbits} of flats to a summation over flats (there are $[W:N(L)]$-many flats conjugate to $L$). After canceling the $|N(L)|$ terms and multiplying out by $|W|$ this is precisely the second recursion; the proof is complete.
\end{proof}

\begin{remark}
The simple lemma above is used twice in this paper (Thm.~\ref{Thm: main} and Sec.~\ref{Sec: Fomin-Reading}) but we believe it is interesting on its own (it encodes some of the geometry of the Coxeter complex). For instance, the following two equations are equivalent because of it but only the first one is trivial: $$1=\dfrac{1}{n}\cdot\sum_{s\in S}1 \quad\quad\text{ versus }\quad\quad n\cdot |W|=2\cdot \sum_{L\in\mathcal{L}_W^1}|W_L|.$$
\end{remark}

\begin{remark}
There is of course a similar version of this Lemma for recursions that run through arbitrary \emph{subsets} of simple reflections (i.e. recursions over \emph{standard} parabolic subgroups) on the left hand side and over any flats (i.e. recursions over all parabolic subgroups) on the right hand side.
\end{remark}

\subsection{Coxeter elements, Coxeter numbers, and the noncrossing partition lattice}
\label{Sec: Cox elts, NC(W)}

Given a system of simple generators $S:=\{s_i\}_{i=1}^n$ of $W$, the product $s_1\cdots s_n$, in \emph{any} order, is called a Coxeter element for $W$. Coxeter made some remarkable observations on these products already in \cite{coxeter-annals}; they are all conjugate [ibid, Thm.~11] and their common order $h$ determines the number $N$ of reflections of $W$ (so that $N=|\mathcal{A}|$) via the formula [ibid, Thm.~16] \begin{equation}2N=hn.\label{Eq: 2N=hn}\end{equation} 
Any element conjugate to one of these products $s_{i_1}\cdots s_{i_n}$ is a Coxeter element for a conjugate simple system. We therefore adopt the following definition.

\begin{definition}
For a real reflection group $W$, a {\color{blue} \emph{Coxeter element}} $c$ is defined as any element conjugate to some product $s_{i_1}\cdots s_{i_n}$ (in any order) of the simple generators $s_i$ of $W$. The Coxeter elements of $W$ form a single conjugacy class and their  common order $h=|c|$ is called the {\color{blue} \emph{Coxeter number}} of $W$.
\end{definition}

Coxeter's original proof \cite[Thm.~16]{coxeter-annals} of formula \eqref{Eq: 2N=hn} was based on the classification of real reflection groups. A few years later Steinberg gave a uniform argument for it \cite[4.2]{stein_hn} which we briefly review as its ideas led to the Deligne-Reading recursion \eqref{eq intro DR recursion} we discussed in the introduction. After fixing a simple system $S$ for $W$, we can color the vertices of the corresponding Coxeter diagram red and blue, so that adjacent vertices have different colors. Then the product $s_{i_1}\cdots s_{i_k}\cdot s_{j_1}\cdots s_{j_{n-k}}$ where the $s_i$'s are red and the $s_j$'s blue is known as a {\color{blue} \emph{bipartite Coxeter element}}. Steinberg gave an explicit construction of a $2$-dimensional plane $P$ (now known as the \emph{Coxeter plane}) on which a bipartite Coxeter element $c$ acts as a rotation of order $h$. All reflection hyperplanes of $W$ intersect $P$ transversely and determine on it the rank-$2$ arrangement of the dihedral group $I_2(h)$ (i.e. with $h$-many lines through the origin). In this way, these lines of $P$ form a partition of the reflections of $W$ and in fact any two consecutive lines correspond to a system of simple generators for $W$. This is sufficient to prove formula \eqref{Eq: 2N=hn}, but in fact we can get the slightly stronger version below.

\begin{proposition}[{\cite{stein_hn}}]\label{Prop: action on R by bipart. c}
For a real reflection group $W$ with Coxeter number $h$ and a choice of simple generators $S$, the action by conjugation of a bipartite Coxeter element $c\in W$ divides the set of reflections $\mathcal{R}$ of $W$ into orbits, each of which has either size $h/2$ and contains a single simple generator from $S$ or it has size $h$ and contains two simple generators of $S$.
\end{proposition}

\begin{remark}
Since the Coxeter elements form a single conjugacy class in $W$, they must define the same orbit structure on $\mathcal{R}$. However, the finer structure regarding the intersections of these orbits with the simple system $S$ does not hold for arbitrary Coxeter elements (the property of being bipartite is sufficient). For example, in the symmetric group $S_n$ with set of simple generators $S:=\{(12),\ldots ,(n-1,n)\}$, the element $c:=(12\cdots n)$ is a Coxeter element but not bipartite with respect to $S$; here in fact the simple generators all lie in the same orbit under $c$ (which further contains a last element $(1n)$). 
\end{remark}

The set of simple generators $S$ of $W$ determines a natural length function $\ell_S$ on $W$ which then can be used to define the Bruhat order of $W$. In a similar fashion we may consider the set of all reflections $\mathcal{R}$ of $W$ as a generating set and define the (absolute) {\color{blue} \emph{reflection length} $\ell_{\mathcal{R}}(w)$} of an element $w\in W$ as the smallest number $k$ for which we can write a factorizations $w=t_1\cdots t_k$ in reflections $t_i\in\mathcal{R}$. Then we have a natural order $\leq_{\mathcal{R}}$ on $W$ where $u\leq_{\mathcal{R}} v$ if and only if $\ell_{\mathcal{R}}(u)+\ell_{\mathcal{R}}(u^{-1}v)=\ell_{\mathcal{R}}(v)$; we call it the {\color{blue} \emph{absolute reflection order}} of $W$.

\begin{definition}\label{Defn: NC(W)}
For a real reflection group $W$ and a Coxeter element $c\in W$ we define the {\color{blue}\emph{noncrossing partition lattice} $NC(W)$} as the interval $[1,c]_{\leq_{\mathcal{R}}}$ in $W$ under the absolute reflection order $\leq_{\mathcal{R}}$. It is a lattice after work of Brady and Watt \cite{BW_lattice} (see also \cite{reading_lattice} for an alternative proof).
\end{definition}

\begin{remark}\label{Rem NC(W) indep of c}
The noncrossing lattice $NC(W)$ with its definition above depends only \emph{superficially} on the choice of Coxeter element $c$. Indeed, if $c'=w^{-1}cw$ is a different Coxeter element, then conjugation by $w$ is a poset isomorphism between the intervals $[1,c]_{\leq_{\mathcal{R}}}$ and $[1,c']_{\leq_{\mathcal{R}}}$ (this is because the set of reflections $\mathcal{R}$ is a union of conjugacy classes).
\end{remark}

As we discussed in the introduction, the noncrossing partition lattice $NC(W)$ has become a central object of interest in the Coxeter combinatorics, representation theory, and geometric group theory of $W$. On the combinatorial side, we will focus on its {\color{blue} \emph{Zeta polynomial}}, a function that counts the number of chains of given lengths in $NC(W)$:
\begin{equation}
\mathcal{Z}\big( NC(W),k\big):=\#\big\{1\leq_{\mathcal{R}} w_1\leq_{\mathcal{R}}\cdots \leq_{\mathcal{R}}w_k\leq_{\mathcal{R}} c\big\}.\label{Eq: Defn Z(NC(W))}
\end{equation}  

The Zeta polynomial of the noncrossing partition lattice has a remarkable formula; before stating it we need to recall some basic facts of the invariant theory of reflection groups $W$ \cite[Ch.~3]{HUM}. The action of $W$ on the ambient space $V\cong\mathbb{R}^n$ induces an action on the polynomial algebra $\CC[V]:=\op{Sym}(V^*)$ via $(w*f)(v):=f(w^{-1}\cdot v)$ and its invariant subalgebra $\CC[V]^W$ can always be generated by a set of algebraically independent, homogeneous polynomials $f_1,\ldots,f_n$, which will be called \emph{fundamental invariants} of $W$. Even though there are different choices for them, the collection of their degrees is uniquely determined. We write $d_i:=\op{deg}(f_i)$, assume that $d_{i+1}\geq d_i$, and call them the {\color{blue}\emph{fundamental degrees}} of $W$. These numbers $d_i$ further determine the eigenvalues of Coxeter elements (they are given as the collection $\{e^{2\pi i\cdot (d_i-1)/h}\},\ i=1,\ldots, n$) and always satisfy $d_1=2$, $d_n=h$, and $d_i+d_{n-i}=2+h$. We are now ready to state Chapoton's formula.

\begin{theorem}[{\cite[Prop.~9]{chapoton}}]\label{Prop: Chapoton's formula}
For an irreducible real reflection group $W$ of rank $n$ with Coxeter number $h$ and fundamental degrees $(d_i)_{i=1}^n$, the Zeta polynomial \eqref{Eq: Defn Z(NC(W))} of its noncrossing partition lattice $NC(W)$ is given as $$\mathcal{Z}\big(NC(W),k\big)=\prod_{i=1}^n\dfrac{kh+d_i}{d_i}.$$
\end{theorem}

A chain as in the right hand side of \eqref{Eq: Defn Z(NC(W))} will be called \emph{saturated} if $w_i\lneq_{\mathcal{R}}w_{i+1}$ for all $i$. The main combinatorial object of this work is the collection of \emph{maximal saturated chains in $NC(W)$}, whose size we call the {\color{blue} \emph{chain number of $W$}} and denote by $MC(W)$:
\begin{equation}
MC(W):=\#\{1\lneq_{\mathcal{R}}w_1\lneq_{\mathcal{R}}\cdots\lneq_{\mathcal{R}}w_{n-1}\lneq_{\mathcal{R}}c\}.\label{Eq: MC(W) as max chains}
\end{equation}Notice that all these chains have length $n$, equal to the rank of $W$, since $\ell_{\mathcal{R}}(c)=n$. It is a standard fact in the theory of posets \cite[Prop.~3.12.1]{EC1} that the leading term of the Zeta polynomial encodes the number of maximal, saturated chains in the poset. Thus we have the following corollary (also recorded in \cite[Prop.~9]{chapoton}).

\begin{theorem}\label{Prop. MC(W)=h^nn!/|W|}
For a reflection group $W$ as in Thm.~\ref{Prop: Chapoton's formula} the chain number of $W$ is given as $$MC(W)=\dfrac{h^nn!}{|W|}.$$
\end{theorem}

It is easy (and we will need it in Section~\ref{Sec: The proof}) to obtain a version of Thm.~\ref{Prop. MC(W)=h^nn!/|W|} when $W$ is reducible. Assume in that case that we can write $W=W_1\times W_2\times\cdots\times W_s$, where the $W_i$ are irreducible real reflection groups with Coxeter numbers $h_i$ and ranks $r_i$ such that $r_1+\cdots+r_s=n$. Then, the maximal chains in $NC(W)$ are shuffles of maximal chains of the lattices $NC(W_i)$ and we have $$MC(W)=\binom{n}{r_1,\ldots,r_s}\cdot MC(W_1)\cdots MC(W_s).$$
This immediately implies the formula $\displaystyle MC(W)=\dfrac{n!}{|W|}\cdot \prod_{i=1}^sh_i^{r_i}$. This is of course a complete answer but, to avoid always keeping track of the various components and their ranks, we will introduce a new notation. For $W$ as before, we define the {\color{blue} \emph{multiset of Coxeter numbers of $W$}} as the collection
\begin{equation}\{h_i(W)\}_{i=1}^n:=\{\underbrace{h_1,\dots,h_1}_{r_1\text{-times}},\dots,\underbrace{h_s,\dots,h_s}_{r_s\text{-times} }  \}.\label{Eq: Defn of multi Cox}\end{equation}
For instance, we have $\{h_i(A_3)\}=\{4,4,4\}$ and $\{h_i(A_2\times B_3\times F_4)\}=\{3,3,6,6,6,12,12,12,12\}$. We will see in Corol.~\ref{Corol: det(L_W)=prod h_i(W)} that the multiset of Coxeter numbers of $W$ is in fact the collection (with multiplicity) of eigenvalues of the $W$-Laplacian matrix $L_W$. With this new notation, we have shown the following.

\begin{corollary}\label{Cor. irr to red}
For a possibly reducible real reflection group $W$ of rank $n$ and with multiset of Coxeter numbers $\big\{h_i(W)\big\}_{i=1}^n$, the chain number $MC(W)$ of \eqref{Eq: MC(W) as max chains} is given as $$MC(W)=\dfrac{n!\cdot\prod_{i=1}^nh_i(W)}{|W|}.$$
\end{corollary}

Chapoton's formula in Thm.~\ref{Prop: Chapoton's formula} is uniformly stated but its proof relies on the classification of real reflection groups. Until very recently (\cite{Douvr}, and previously  \cite{Michel} for Weyl groups, but see Remark~\ref{Rem: Advantages of our proof}), its corollary in Thm.~\ref{Prop. MC(W)=h^nn!/|W|} could only be derived after case-by-case arguments as well. In Section~\ref{Sec: The proof} we give a proof of Thm.~\ref{Prop. MC(W)=h^nn!/|W|} that is case-free and relies only on Coxeter-theoretic considerations.

\section{The $W$-Laplacian and a parabolic recursion on Coxeter numbers}
\label{Sec: W-Laplacian}

In our previous work \cite[\S~3.4]{chapuy_douvr} we introduced for any (complex) reflection group $W$ an $(n\times n)$ matrix, called the $W$-Laplacian, which generalized the usual, weighted Laplacian matrix $L(K_n,\bm\omega)$ of the complete graph $K_n$. The $W$-Laplacian encodes via its spectrum a variety of structural and combinatorial results; the weighted enumeration of (arbitrary length) reflection factorizations of Coxeter elements [ibid, Thm.~1], the volume calculation of root zonotopes [ibid, Thm.~8.11], and a generalization of the Matrix-forest theorem for reflection groups [ibid, Thm.~3]. An important ingredient of this theory was a parabolic recursion for the characteristic polynomial of some more general $\mathcal{A}$-Laplacian matrices (for hyperplane arrangements $\mathcal{A}$) that was analogous to Brieskorn's lemma \cite{briesk_lemma} (or for a textbook approach, see \cite[Corol.~3.9]{dimca}). 

In this section, we review the definition of the $W$-Laplacian matrix, and reproduce this parabolic recursion in Lemma~\ref{Lem: parab. rec. det(L_W+t)}; in fact with a more direct proof. This recursion is the main ingredient through which we give a uniform derivation of the formula for the chain numbers $MC(W)$ in Theorem~\ref{Thm intro chapoton}; we do this by compairing it with the Deligne-Reading recursion \eqref{eq intro DR recursion}. Even though we only deal with the unweighted case and with reflection arrangements here, the proof of the technical Lemma~\ref{Lem: parab. rec. det(L_W+t)} applies essentially verbatim to the same generality as in \cite[Prop.~8.3]{chapuy_douvr}. 

\begin{definition}[The $W$-Laplacian]\label{Defn: the W-Laplacian}
Let $W$ be a real reflection group of rank $n$ acting on the space $V\cong \mathbb{R}^n$ with reflection representation $\rho^{}_V$. We denote its set of reflections by $\mathcal{R}$ and positive root system by $\Phi^{+}$; we treat $\Phi^{+}$ as a subset of $V$ and we normalize all roots $\sigma\in\Phi^{+}$ so that $\langle \sigma,\sigma\rangle=2$ for the standard $W$-invariant inner product $\langle\cdot,\cdot\rangle$. Then, we define the {\color{blue}\emph{$W$-Laplacian matrix} $L_W$} as
$$\op{GL}(V)\ni L_W:=\sum_{\tau\in\mathcal{R}}\big(\mathbf{I}_n-\rho_V^{}(\tau)\big)\quad\text{or equivalently}\quad L_W(v):=\sum_{\sigma\in\Phi^{+}}\langle \sigma, v\rangle\cdot\sigma,$$
where $\mathbf{I}_n$ denotes the $(n\times n)$ identity matrix. The equivalence of the two definitions is clear as we have $\big(\rho_V^{}(\tau)\big)(v)=v-\langle\sigma,v\rangle\cdot \sigma$ when $\sigma$ is any of the two roots that correspond to the reflection $\tau$.  
\end{definition}

\begin{remark}
It is not too hard to see that for the symmetric group $S_n$ acting on $\mathbb{R}^n$, the $W$-Laplacian agrees with the usual graph Laplacian $L(K_n)$ of the complete graph $K_n$. We display, with contracted notation, the case $n=4$ below.
$$\underbrace{\begin{bmatrix}
\ \ 3& -1 & -1 & -1\\
-1 &\ \  3 & -1 & -1\\
-1 & -1 &\ \  3 & -1\\
-1 & -1 & -1 &\ \  3 \end{bmatrix}}_{L(K_4)} = \sum_{i<j}\begin{bmatrix}
0 &\ \ 0 & \ \ 0 & 0 \\
0 &\ \ 1 & -1 & 0\\
0 & -1 & \ \ 1 & 0\\
0 & \ \ 0 &\ \ 0 & 0\end{bmatrix}=\sum_{i<j} \cdot \Bigg(\mathbf{I}_4- \underbrace{\begin{bmatrix}
1&0&0&0\\
0&0&1&0\\
0&1&0&0\\
0&0&0&1\end{bmatrix}}_{\rho^{}_V\big((ij)\big)}\Bigg).
$$
Moreover, the other common presentation of the graph Laplacian as $L(G)=M\cdot M^T$, where $M$ is the \emph{oriented} incidence matrix of the graph $G$, has again a natural analog here. If $(e_i)_{i=1}^n$ is an orthonormal basis for $\mathbb{R}^n$ and $R$ denotes the $(n\times N)$ matrix $R_{ij}:=\langle e_i,\sigma_j\rangle$ where $\sigma_j\in\Phi^{+}$, it is immediate from the second definition above that the matrix of $L_W$ in the $(e_i)$-basis equals $R\cdot R^T$.
\end{remark}

We will see now the first and most fundamental connection between the $W$-Laplacian $L_W$ and the reflection group $W$; the spectrum of $L_W$ is given via the Coxeter numbers. This observation has deep connections to representation theory initiated perhaps by Beynon-Lustig, generalized by Gordon-Griffeth, and pursued by many others \cite{BL,Malle,GG}. The first appearance of the following statement however seems to go back to Bourbaki \cite[Ch.~5:~\S~6]{bourbaki}. Their beautiful argument directly relates it to the formula \eqref{Eq: 2N=hn} so we reproduce it here (see \cite[Prop.~3.13]{chapuy_douvr} for the complex case).

\begin{proposition}\label{Prop: L_W is mult by h}
For an irreducible real reflection group $W$ with Coxeter number $h$, the $W$-Laplacian $L_W$ is the matrix of scalar multiplication by $h$.
\end{proposition}
\begin{proof}
After rewriting the second version of Defn.~\ref{Defn: the W-Laplacian} as $\displaystyle L_W(v):=\dfrac{1}{2}\sum_{\sigma\in \Phi}\langle\sigma,v\rangle\cdot\sigma$, where now the sum is over all roots, it is immediately clear that $L_W$ is a $W$-equivariant map on $V$. Since moreover $W$ is assumed irreducible, we know by Schur's lemma that $L_W$ must act on $V$ as multiplication by some scalar $\beta$. That is, we will have $$\sum_{\sigma\in\Phi^{+}}\langle\sigma,v\rangle\cdot\sigma=\beta\cdot v.$$ Then, if $n$ is the rank of $W$ and $(e_i)_{i=1}^n$ an orthonormal system of coordinates for $V$, we can write $$\beta\cdot n=\beta\cdot\sum_{i=1}^n\langle e_i,e_i\rangle= \sum_{i=1}^n\langle \beta\cdot e_i,e_i\rangle=\sum_{i=1}^n\sum_{\sigma\in\Phi^{+}}\langle \sigma, e_i\rangle^2=\sum_{\sigma\in \Phi^{+}}\langle \sigma,\sigma\rangle=2\cdot |\Phi^{+}|.$$ That is, we have $\beta=2N/n$ which is equal to $h$ after \eqref{Eq: 2N=hn}.
\end{proof}

\begin{corollary}\label{Corol: det(L_W)=prod h_i(W)}
For a possibly reducible real reflection group $W$, the determinant of the $W$-Laplacian is given as $$\op{det}(L_W)=\prod_{i=1}^n h_i(W),$$ where $\{h_i(W)\}$ is the \emph{multiset} of Coxeter numbers of $W$ as defined in \eqref{Eq: Defn of multi Cox}.
\end{corollary}
\begin{proof}
Suppose that $W$ decomposes into irreducibles as $W=W_1\times\cdots\times W_s$. Then, the set of reflections of $W$ is made up of the elements $(1,\cdots,\tau,\cdots,1)$ where $\tau$ is a reflection of some irreducible component $W_i$. This implies, after Defn.~\ref{Defn: the W-Laplacian}, that we can write the $W$-Laplacian as $$L_W=\bigoplus_{i=1}^s L_{W_i},$$ where $\bigoplus$ denotes the direct sum of matrices. Now the required statement is a direct corollary of Prop.~\ref{Prop: L_W is mult by h} and the definition of the multiset of Coxeter numbers as in \eqref{Eq: Defn of multi Cox}.
\end{proof}

The following statement is the key technical lemma of the paper. Its proof relies on the fact that the $W$-Laplacian is a sum of rank 1 operators (as in the second part Defn.~\ref{Defn: the W-Laplacian}). Burman et al. \cite{burman_mt} studied such and wider classes of operators giving an abstract matrix tree theorem for their characteristic polynomials and in fact Lemma~\ref{Lem: parab. rec. det(L_W+t)} can be deduced from theirs [ibid, Corol.~2.4] without too much extra work. In \cite[Lemma~8.2\,\&\,Prop.~8.3]{chapuy_douvr} we gave a proof of this statement using the Cauchy-Binet from linear algebra; we present here (also for completeness) a more direct approach.

\begin{lemma}\label{Lem: parab. rec. det(L_W+t)}
For an irreducible real reflection group $W$ with reflection arrangement $\mathcal{A}$, the characteristic polynomial of its $W$-Laplacian satisfies the following recursion.
$$\op{det}(t\cdot \mathbf{I}_n+L_W)=\sum_{X\in\mathcal{L}_{\mathcal{A}}}\op{pdet}(L_{W_X})\cdot t^{\op{dim}(X)},$$ where the \emph{pseudodeterminant} $\op{pdet}(A)$ is defined as the smallest degree nonzero coefficient of the characteristic polynomial $\op{det}(t\cdot \mathbf{I}+A)$ of the matrix $A$.
\end{lemma}
\begin{proof}
Recall that given a matrix $M\in\op{GL}(V)$, its $k$-th exterior power $\bigwedge^k(M)$ is defined as the matrix which acts on $\bigwedge^k(V)$ via (the linear extension of) $$\bigwedge^k(M)(v_1\wedge\cdots\wedge v_k)=M(v_1)\wedge\cdots\wedge M(v_k),$$ for any elements $v_i\in V$. Then, it is a standard fact from linear algebra that the coefficients of the characteristic polynomial of $M$ are given via the traces of these exterior powers (up to an alternating sign). In our case of the $W$-Laplacian if we expand its characteristic polynomial (evaluated at $-t$ to avoid the alternating signs) into the expression
$$\op{det}(t\cdot \mathbf{I}_n+L_W)=t^m+c_{m-1}t^{m-1}+\cdots+c_1\cdot t+c_0,$$ then the coefficients are given as $c_{n-k}=\op{Tr}\big(\bigwedge^k(L_W)\big)$.

Now, the $k$-th exterior power of the $W$-Laplacian acts on an element $v_1\wedge \cdots \wedge v_k$ of the ambient space $V$ as follows:
\begin{align*}
\bigwedge^k(L_W)(v_1\wedge\cdots\wedge v_k)&=L_W(v_1)\wedge\cdots \wedge L_W(v_k)=\Big(\sum_{\sigma\in\Phi^{+}}\langle\sigma,v_1\rangle\cdot\sigma\Big)\wedge\cdots\wedge \Big(\sum_{\sigma\in\Phi^{+}}\langle\sigma,v_k\rangle\cdot \sigma\Big)\\
&=\sum_{(\sigma_{i_1},\sigma_{i_2},\ldots,\sigma_{i_k})\in (\Phi^{+})^k}\langle\sigma_{i_1},v_1\rangle\cdots\langle\sigma_{i_k},v_k\rangle\cdot \sigma_{i_1}\wedge\cdots\wedge\sigma_{i_k},
\end{align*}where the sum is over all $k$-tuples of positive roots. For the tuple $(\sigma_{i_1},\ldots,\sigma_{i_k})$ to make a non-trivial contribution to the above sum, the roots $\sigma_{i_j}$ must be linearly independent. If they are, then they span the space $X^{\perp}$ orthogonal to the flat $X\in\mathcal{L}_{\mathcal{A}}$, of dimension $n-k$, determined as the intersection of their hyperplanes $H_j:=(\sigma_{i_j})^{\perp}$. This further implies that all $\sigma_{i_j}$ belong to the root system $\Phi^{+}_{W_X}$ of the parabolic subgroup $W_X$. If we group together the terms that appear in the sum above, \emph{with respect to this flat $X$}, we have the following.
\begin{align*}\bigwedge^k(L_W)(v_1\wedge\cdots \wedge v_k)&=\sum_{X\in\mathcal{L}_{\mathcal{A}}}\ \sum_{(\sigma_{i_1},\ldots,\sigma_{i_k})\in(\Phi^{+}_{W_X})^k}\langle\sigma_{i_1},v_1\rangle\cdots\langle\sigma_{i_k},v_k\rangle\cdot \sigma_{i_1}\wedge\cdots\wedge\sigma_{i_k}\\
&=\sum_{X\in\mathcal{L}_{\mathcal{A}}}\bigwedge^k(L_{W_X})(v_1\wedge\cdots\wedge v_k), \end{align*}where the second equality is by the same argument as previously. Now, by construction (see Defn.~\ref{Defn: the W-Laplacian}), the Laplacian $L_{W_X}$ is an $(n\times n)$ matrix of rank $k$. Therefore, the first nonzero coefficient of its characteristic polynomial appears in degree $k$ and, as in the beginning, it is equal to the trace of $\bigwedge^k(L_{W_X})$. The proof is now complete.
\end{proof}

The following theorem is now an immediate corollary of this parabolic recursion of Lemma~\ref{Lem: parab. rec. det(L_W+t)}. It also holds for arbitrary complex reflection groups as we have shown in \cite[Thm.~8.8]{chapuy_douvr}. Recall the definition of the \emph{multiset of Coxeter numbers} $\{h_i(W)\},\, i=1,\ldots,n$ from \eqref{Eq: Defn of multi Cox}.

\hrecursion
\begin{proof}
Indeed, this is a direct combination of Lemma~\ref{Lem: parab. rec. det(L_W+t)} and Corol.~\ref{Corol: det(L_W)=prod h_i(W)}. The appearance of the pseudodeterminant in the Lemma but not the Corollary should not be confusing, it is because we have silently assumed in Corol.~\ref{Corol: det(L_W)=prod h_i(W)} that $W$ is essential, that is, it does not fix (pointwise) a proper linear subspace. The parabolic subgroups $W_X$ are of course not essential but they do act as such on the orthogonal complements $X^{\perp}$. Their Laplacians $L_{W_X}$ differ by those of the essentializations only by a direct sum with a $k\times k$ zero matrix (where $k=\op{dim}(X)$), hence we have that the product of the (multiset) Coxeter numbers $h_i(W_X)$ equals the \emph{pseudodeterminant} of $L_{W_X}$. 
\end{proof}

\newcommand{\myPic}[1]{\begin{minipage}[b]{4mm}\begin{center}{\includegraphics[scale=0.4]{#1}}\end{center}\end{minipage}}
\begin{example}
We give here, as an example, the statement of Thm.~\ref{Thm: rec cox nums} for the symmetric group $S_4$. In this case, the intersection lattice $\mathcal{L}_{\mathcal{A}}$ is the lattice of set partitions of $[4]:=\{1,2,3,4\}$ and the parabolic subgroups $W_X$ are just the corresponding Young subgroups. 

There are $4$ partitions conjugate to $123|4$, corresponding to parabolic subgroups isomorphic to $S_3$ with multiset of Coxeter numbers $\{3,3\}$. Similarly, $3$ partitions conjugate to $12|34$ each leading to the multiset $\{2,2\}$. Then, $6$ partitions of the form $12|3|4$ corresponding to the six transpositions of $S_4$ and with multiset $\{2\}$. Adding to these the whole group $S_4$ with multiset of Coxeter numbers $\{4,4,4\}$ and the trivial subgroup $S_1$ giving the empty set, we have the following equation.
$$(t+4)^3=\underbrace{4^3}_{\myPic{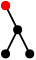}}+\big(\underbrace{4\cdot 3^2}_{\hspace{-6mm}\myPic{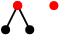}}+ \underbrace{3\cdot(2\cdot 2)}_{\myPic{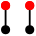}}\big)\cdot t+\underbrace{6\cdot 2}_{\hspace{-4mm}\myPic{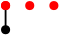}}\cdot\, t^2+\underbrace{1}_{\hspace{-5mm}\myPic{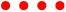}}\cdot\, t^3.$$
Notice that we have indexed the terms of the equation by (isomorphism types of) rooted forests. Indeed, for the symmetric group $S_n$, Theorem~\ref{Thm: rec cox nums} reflects the Matrix-Forest theorem and the summation term associated to a set partition $X$ counts rooted forests whose number of trees and tree sizes are encoded in $X$. For more details see our previous work \cite[\S~8]{chapuy_douvr}.
\end{example}

\section{A simple uniform proof of the chain number formula for $NC(W)$}
\label{Sec: The proof}

In this section we present the main contribution of this work; a case-free proof (Thm.~\ref{Thm: main}) of the Arnold-Bessis-Chapoton formula for the number $MC(W)$ of maximal chains in the noncrossing lattice $NC(W)$ of a real reflection group $W$. We start by briefly reviewing from \cite{reading_chains} the first ingredient of our argument, the Deligne-Reading recursion for $MC(W)$ (stated as Prop.~\ref{Prop: Del-Read rec}). 

It is clear after Defn.~\ref{Defn: NC(W)} that a \emph{maximal} chain in $NC(W)$ corresponds to a \emph{reduced} (i.e. of minimum length, which then has to equal the rank $n$ of $W$) reflection factorization $t_1\cdots t_n=c$ of the Coxeter element $c$. That is, we may rewrite \begin{equation}
MC(W)=\#\{(t_1,\ldots,t_n)\in \mathcal{R}^n\ |\ t_1\cdots t_n=c\},\label{Eq: MC(W) as Red_R(c)}
\end{equation}where $\mathcal{R}$ denotes as usual the set of reflections of $W$.  Now, we may act on the set of such reduced reflection factorizations by conjugating all terms by $c$, since again $ ^ct_1\cdots ^ct_n=c$ if, as usual, we write $ ^ct_i:=c^{-1}\cdot t_i\cdot c$. The orbits of this action have size $h/2$ when $c^{h/2}=-1$ and size $h$ otherwise (because the terms $t_i$ of such a factorization must generate $W$ and the only power of $c$ that can lie in the center of $W$ is $c^{h/2}$). 

Let us assume at this point that $c$ is a bipartite Coxeter element with respect to the simple generators $S$ of $W$ (we can safely do that since $MC(W)$ is independent of the choice of $c$, see Remark~\ref{Rem NC(W) indep of c}). Then, after Prop.~\ref{Prop: action on R by bipart. c}, there are two possible cases for the orbits described above. They either have length $h/2$ and contain a single factorization of the form $t_1'\cdots t_{n-1}'\cdot s_i=c$ for some $s_i\in S$, or they have length $h$ and contain two \emph{different} factorizations  $t_1'\cdots t_{n-1}'\cdot s_i=c$ and $t_1''\cdots t_{n-1}''\cdot s_j=c$, where $s_i$ and $s_j$ are not \emph{necessarily} different however.

Because $c$ is bipartite, all elements $c\cdot s_i$ are either equal or conjugate to some Coxeter element of the maximal parabolic subgroup $W_{\langle s_i\rangle}$. Indeed, it turns out that for any $s_i\in S$ there is always an expression of $c$ as a product of simples that either starts or ends with $s_i$ (this relies on the bipartite structure of $c$ as described in \S~\ref{Sec: Cox elts, NC(W)}). Now, enumerating all factorizations in \eqref{Eq: MC(W) as Red_R(c)} with respect to their $c$-orbit as described in the previous paragraph, we immediately get the following statement.

\begin{proposition}[{The Deligne-Reading recursion \cite[Corol.~3.1]{reading_chains}}]\label{Prop: Del-Read rec}\ \newline 
For an irreducible real reflection group $W$ with set of simple generators $S$ and Coxeter number $h$, the number $MC(W)$ of maximal chains in the noncrossing lattice $NC(W)$ satisfies the recursion
$$MC(W)=\dfrac{h}{2}\cdot \sum_{s\in S}MC(W_{\langle s\rangle}),$$ where $W_{\langle s\rangle}$ is the (standard) parabolic subgroup of $W$ generated by the simple reflections $S\setminus\{s\}$.
\end{proposition}

\subsection{The main proof} We will now prove Thm.~\ref{Thm: main} by showing that the right hand side of the formula of Thm.~\ref{Prop. MC(W)=h^nn!/|W|} satisfies the same parabolic recursion as Prop.~\ref{Prop: Del-Read rec} above. Again, we note that the novel contribution here is its proof, which is both uniform and solely Coxeter-theoretic.

\begin{theorem}\label{Thm: main}
For a (possibly reducible) real reflection group $W$ of rank $n$ and with multiset of Coxeter numbers $\{h_i(W)\}_{i=1}^n$, the number $MC(W)$ of maximal chains in its noncrossing lattice $NC(W)$ is given by the formula $$MC(W)=\dfrac{n!}{|W|}\cdot\prod_{i=1}^nh_i(W).$$In particular, if $W$ is irreducible with Coxeter number $h$, we have $MC(W)=(h^n\cdot n!)/|W|$.
\end{theorem}
\begin{proof}
The proof proceeds by induction. The case of the symmetric group $S_2$, which is the unique reflection group of rank $n=1$, is trivial; $NC(W)$ is a lattice with only two elements and thus a single maximal chain. Indeed, we have $MC(S_2)=1=(2^1\cdot 1)/2!$.

Let us assume now that the theorem is proven for all reflection groups of rank at most $n-1$. We only have to prove the statement for \emph{irreducible} groups of rank $n$ as the argument of Corol.~\ref{Cor. irr to red} extends the result (trivially and uniformly) to the reducible case. 

Assuming $W$ is an irreducible real reflection group of rank $n$ and Coxeter number $h$, Lemma~\ref{Lem: simples to flats} allows us then to rewrite the Deligne-Reading recursion of Prop.~\ref{Prop: Del-Read rec} as $$|W|\cdot MC(W)=h\cdot \sum_{L\in\mathcal{L}^1_{\mathcal{A}}}MC(W_L)\cdot|W_L|,$$ where the sum is over all lines (1-dimensional flats) of the reflection arrangement $\mathcal{A}$ of $W$. We want to show that $|W|\cdot MC(W)=h^n\cdot n!$ which, after applying the inductive assumption on the terms $MC(W_L)$ above, becomes equivalent to $$h^n\cdot n!=h\cdot\sum_{L\in\mathcal{L}^1_{\mathcal{A}}}(n-1)!\cdot\prod_{i=1}^{n-1}h_i(W_L),$$ or, cancelling the $(n-1)!$ term and the single factor $h$ of the right hand side, \begin{equation}
h^{n-1}\cdot n=\sum_{\mathcal{L}\in\mathcal{L}^1_{\mathcal{A}}}\prod_{i=1}^{n-1}h_i(W_L).\label{EQ: h^{n-1}-h_i(W_L)}
\end{equation} This is a direct corollary of Theorem~\ref{Thm: rec cox nums} by comparing the coefficients of $t^1$ in both sides of the equation. The proof is complete. 
\end{proof}

\begin{remark}\label{Rem: Advantages of our proof}
As we mentioned in the introduction, a generalization of Theorem~\ref{Thm: main} due to Stump and the first author \cite{chapuy_stump} has recently received uniform proofs, one in the case of Weyl groups by Jean Michel \cite{Michel} and one for all real reflection groups (which also applies to the complex well-generated case but is not completely uniform there) by the second author \cite{Douvr}. However, both these proofs are quite technical and rely on the representation theory of $W$ and some highly non-trivial manipulation of its characters. 

The proof we present above is case-free and completely elementary (or at least Coxeter-theoretic). It reveals the utility of the $W$-Laplacian in Coxeter combinatorics (and adjacent areas, see \S~\ref{Sec: implications braid group}).
\end{remark}

\subsection{The case of complex reflection groups}
\label{Sec: The well-generated case}
Coxeter elements exist also in the family of \emph{well generated} complex reflection groups, and our derivation of the chain number formula of Thm.~\ref{Thm: main} extends to them as well (and this is already known, via a case-by-case check, by \cite[Prop.~7.6]{Bessis} or \cite{chapuy_stump}). However, in this setting our proof is also not completely uniform and moreover relies eventually on the difficult geometry of the Lyashko-Looijenga morphism, so we only briefly sketch it here.

The Deligne-Reading recursion may be replaced by a simpler one which only keeps track of the last reflection in a reduced factorization of the Coxeter element. Then the recurrence on the chain numbers becomes
\[
MC(W)=\sum_{[L]\in\mathcal{L}^1_{\mathcal{A}}/W}\op{Krew}^{}_W(L)\cdot MC(W_L),
\]
where the \emph{Kreweras number} $\op{Krew}_W(L)$ counts the noncrossing partitions whose fixed space is conjugate to the flat $L$. This is easily seen to be equivalent to the expression in \eqref{EQ: h^{n-1}-h_i(W_L)} because of the following formula (see \cite[Prop.~83, 85]{theo_thesis}) for Kreweras numbers of lines:
\begin{equation}
\op{Krew}_W^{}(L)=\dfrac{h}{[N(L):W_L]},\label{EQ: krew nums}
\end{equation}
where $N(L)$ is the setwise normalizer of $L$. Moreover, our Theorem~\ref{Thm: rec cox nums} holds as is for arbitrary complex reflection groups, as we have shown in \cite[\S~8]{chapuy_douvr}, so that the proof of Thm.~\ref{Thm: main} for well generated groups $W$ proceeds in essentially the same way. 

The difference is that the formula in \eqref{EQ: krew nums} relies \cite[Prop.~85]{theo_thesis} on the transitivity of the \emph{Hurwitz action} on reduced reflection factorizations of a Coxeter element, a fact still known only after case-by-case calculations \cite[Prop.~7.6]{Bessis}.

\section{The Fomin-Reading recursion}
\label{Sec: Fomin-Reading}


So far, our techniques were applied only to the chain number $MC(W)$ and did not seem able to produce the formula of Thm.~\ref{Prop: Chapoton's formula} for the Zeta polynomial of $NC(W)$. One part of them however, the Deligne-Reading recursion, has a meaningful analog which we briefly discuss here. 

The number $\mathcal{Z}\big(NC(W),k\big)$ of length-$k$ chains in $NC(W)$ also counts (and this fact has uniform proofs \cite{tzanaki,cataland}) the maximal facets of the \emph{generalized cluster complex} $\Delta^{(k)}(W)$ associated to $W$ (see \cite{FR}). Now, Fomin-Reading have defined a cyclic action on $\Delta^{(k)}(W)$, which translates to the following recursion \cite[Prop.~8.3]{FR} for the Zeta polynomials (and does so in a fashion that subsumes the Deligne-Reading recursion). If $h$ is the Coxeter number of $W$ and $n$ its rank, we have
\[
\mathcal{Z}\big(NC(W),k\big)=\dfrac{kh+2}{2n}\cdot\sum_{s\in S}\mathcal{Z}\big(NC(W_{\langle s\rangle},k\big).
\]
One could then give a uniform proof of Chapoton's formula of Thm.~\ref{Prop: Chapoton's formula} by proving the resulting recurrence on Coxeter numbers and fundamental degrees. Furthermore, applying Lemma~\ref{Lem: simples to flats} this would mean to prove the following equation
\begin{equation}
\prod_{i=1}^n(kh+d_i)=\dfrac{kh+2}{n}\cdot\sum_{L\in\mathcal{L}^1_{\mathcal{A}}}\prod_{i=1}^{n-1}\big(kh_i(W_L)+d_i(W_L)\big),\label{EQ: FR1}
\end{equation}
where $\{d_i\}$ are the fundamental degrees of $W$, while $\{d_i(W_L)\}$ are those of the parabolic subgroup $W_L$ and the indexings $h_i(W_L)$ and $d_i(W_L)$ are compatible in the obvious way.

Moreover, the previous formula has a generalization for flats of arbitrary dimension $r$, similarly to how Thm.~\ref{Thm: rec cox nums} generalizes \eqref{EQ: h^{n-1}-h_i(W_L)}. It is shown in \cite{BDJ} \emph{but in a case-by-case way}, after a double counting of faces and facets in $\Delta^{(k)}(W)$, that
\begin{equation}
\binom{n}{r}\prod_{i=1}^n(kh+d_i)=\sum_{X\in\mathcal{L}^r_{\mathcal{A}}}\Big(\prod_{i=1}^r(kh+b_i^X+1)\Big)\cdot\Big(\prod_{i=1}^{n-r}(kh_i(W_X)+d_i(W_X)\Big),\label{EQ: FR2}
\end{equation}
where the numbers $b_i^X$ are the \emph{Orlik-Solomon exponents} (see \cite{OS_nu}) associated to the flat $X$. 

We have not managed to give a uniform proof of the Coxeter-theoretic relations \eqref{EQ: FR1} or \eqref{EQ: FR2} above but we hope this will be possible in the future. One disadvantage we must note however is that these formulas are no longer true for all complex well generated groups (while \eqref{EQ: h^{n-1}-h_i(W_L)} does remain true). In other words, this approach could only lead to a proof of Chapoton's formula uniformly for \emph{real reflection groups}.

\section{Implications for the geometric group theory of Artin-Tits groups}
\label{Sec: implications braid group}

Even though Thm.~\ref{Thm: main} is of an enumerative nature, the seminal work of Bessis \cite{Bessis} relates it with deep results on the geometric group theory of finite Coxeter groups $W$ and their (generalized) braid groups $B(W)$. In fact, it is the last needed ingredient for a \emph{uniform} proof that the Artin and dual-braid presentations of $B(W)$ agree (see \S~\ref{Sec: concordance of presn's}). The concordance of the two presentations has been widely used in recent works; in particular, it was a main ingredient in the proof of the $K(\pi,1)$-conjecture for affine Artin groups by Paolini and Salvetti (see \cite[Rem.~5.4,~Thm.~2.13]{paolini_salvetti}) and in the preceding article of McCammond and Sulway (see proof of \cite[Prop.~10.12]{mccammond_sulway}). However, it was mostly used as a  stepping stone in recursive arguments (dealing with the \emph{spherical} parabolic subgroups of the affine Coxeter groups) and its -at the time- reliance on case by case considerations \emph{might} have gone unnoticed. In this section, we briefly set up some necessary background and we review part of Bessis' work, in order to explain how the uniform proof of Thm.~\ref{Thm: main}  was indeed the only missing ingredient for a uniform proof of the equivalence of the two presentations\footnote{As we mentioned, our paper is only the \emph{second} uniform proof of Thm.~\ref{Thm: main}, but the first Coxeter-theoretic one,~see Remark~\ref{Rem: Advantages of our proof}. However, we believe that we have a good occasion here to clarify the role that Thm.~\ref{Thm: main} plays in Bessis' work for the equivalence of the two presentations.}.

\subsection{The standard presentation for spherical Artin groups}
\label{Sec: Artin-presentation}

Recall first Coxeter's presentation for real reflection groups $W$ of rank $n$ with a system of simple generators $S:=\{s_i\}_{i=1}^n$.
\begin{equation}
W=\langle s_1,\ldots,s_n\ |\ s_i^2=1,\ \underbrace{s_is_js_i\cdots}_{m_{ij}}=\underbrace{s_js_is_j\cdots}_{m_{ij}}\ \rangle,\label{Eq: Coxeter presentation}
\end{equation}
where $m_{ij}-2$ is the number of edges between the vertices $i$ and $j$ of the Coxeter diagram (equivalently $m_{ij}$ is the order of $s_is_j$).

In the case of the symmetric group $S_n=A_{n-1}$, removing the order relations $s_i^2=1$ in \eqref{Eq: Coxeter presentation} gives a presentation for Artin's Braid group $B_n$ on $n$ strands. This was known since Artin's work \cite{artin} but much later Fox and Neuwirth \cite{fox_neu} gave a different proof by realizing $B_n$ as the fundamental group of the configuration space of $n$ points in the plane, or equivalently the space of free orbits under the reflection representation of $S_n$. 

Brieskorn pushed the Fox-Neuwirth approach further by defining the \defn{generalized braid group} (also known as \defn{Artin-Tits group}\footnote{Because a few years earlier Tits \cite{tits} had considered abstract groups defined by similar presentations.}) $B(W)$ as the fundamental group of the space $V^{\op{reg}}/W$ of free orbits under the reflection action of a finite Coxeter group $W$. He proved \cite{briesk_fund} then, relying on the invariant theory of dihedral groups, that $B(W)$ has an Artin-like presentation. Notice that we are using bold face letters $\bm s_i$ to distinguish the generators of $B(W)$.

\begin{theorem}[{\cite{briesk_fund}}]\label{Thm: Artin-Brieskorn presentation}
The generalized braid group $B(W)$ of a finite Coxeter group $W$ of rank $n$ has a presentation given by
\[
B(W)=\langle \bm s_1,\ldots,\bm s_n\ |\ \underbrace{\bm s_i\bm s_j\bm s_i\cdots}_{m_{ij}}=\underbrace{\bm s_j\bm s_i\bm s_j\cdots}_{m_{ij}}\ \rangle,
\]
with formal generators $\bm s_i$ and where the $m_{ij}$ can be read from the Coxeter diagram of $W$ as in \eqref{Eq: Coxeter presentation}.
\end{theorem}

\begin{remark}
All of this holds also for infinite Coxeter groups. One can define abstractly an associated Artin group by a presentation like in Thm.~\ref{Thm: Artin-Brieskorn presentation} and Van der Lek showed in his thesis \cite{lek_thesis} that it agrees with the fundamental group of the space of free orbits in the Tits cone. We are interested here mostly in Artin groups associated to finite, Euclidean, and affine Coxeter groups; they are called respectively \defn{spherical}, \defn{Euclidean}, and \defn{affine Artin groups}.
\end{remark}

\subsection{The dual braid presentation for spherical Artin groups}
\label{Sec: Dual-braid presentation}

In the early 2000's after works by many authors \cite{bi_ko_lee,bessis_dual,brady_watt} a new \emph{dual} theory was built for Artin-Tits groups, where the whole set of reflections is taken as a generating set and a new \emph{dual-braid presentation} is constructed. Birman-Ko-Lee \cite{bi_ko_lee} first dealt with the standard Braid group $B_n$, and used the new presentation to give better solutions for its word and conjugacy problems. Bessis, who was motivated by the existence of finite order automorphisms in the generalized braid groups \cite[\S~3.5]{bessis_dual},  arrived in parallel to the same construction but now defined it for all spherical Artin groups $B(W)$. Brady and Watt \cite{brady_watt} used the dual-braid presentation to construct simplicial complexes that formed $K(\pi,1)$-spaces for the groups $B(W)$.  

The dual braid approach for braid groups $B(W)$ works roughly as follows. One considers a generating set $\mathcal{R}$ of $W$ along with the \emph{length} order $\leq_{\mathcal{R}}$ it determines on $W$ and the interval $[1,c]_{\leq R}$ between the identity and a \emph{Coxeter} element $c\in W$ under this order. Then, the dual braid presentation records all positive relations that are visible in this interval. This context was abstractified and formalized first by Jean Michel, who used the term \emph{generated groups} \cite[Thm.~0.5.2]{bessis_dual} for such structures, and independently by Jon McCammond who called them \emph{interval groups} \cite[Defn.~2.4]{failure}. The version we give below is slightly different from the original \cite[Thm.~2.2.5]{bessis_dual} but it is immediately equivalent to it (see \cite[Prop.~3.3]{failure}).

\begin{theorem}[{\cite[Thm.~2.2.5]{bessis_dual}}]\label{Thm: Bessis dual presentation}
Let $W$ be a finite Coxeter group and consider for each reflection $\tau_i\in W, i=1,\ldots, N$ a formal generator $\bm t_{\tau_i}$. Then, the generalized braid group $B(W)$ has a presentation given by 
\[
B(W)=\langle \bm{t}_{\tau_1},\bm{t}_{\tau_2},\ldots,\bm{t}_{\tau_N}\ |\ \underbrace{\bm{t}_{\tau_{i_1}}\cdots \bm{t}_{\tau_{i_n}}=\bm{t}_{\tau_{i'_1}}\cdots\bm{t}_{\tau_{i'_n}}=\ldots}_{h^nn!/|W|-\text{many words}}\ \rangle,
\]
where each word $\bm{t}_{\tau_{i_1}}\cdots \bm{t}_{\tau_{i_n}}$ that appears in the relations corresponds to a reduced reflection factorization $\tau_1\cdots\tau_n=c$ of some fixed Coxeter element $c$ of $W$.
\end{theorem}

\begin{example}
Let's take for instance the symmetric group $A_2=S_3$, with coxeter element the long cycle $c:=(123)$. There are three reduced factorizations of $c$ in transpositions, namely 
\[
(12)(23)=(123)\quad\quad (23)(13)=(123)\quad\quad (13)(12)=(123).
\]
Then, Thm.~\ref{Thm: Bessis dual presentation} gives us the following presentation for the braid group $B(S_3)=B_3$:
\[
B_3=\langle \bm{t}_{(12)},\bm{t}_{(23)},\bm{t}_{(13)}\ |\  \bm{t}_{(12)}\cdot\bm{t}_{(23)}=\bm{t}_{(23)}\cdot\bm{t}_{(13)}=\bm{t}_{(13)}\cdot\bm{t}_{(12)}\rangle,
\]
which we can rewrite more clearly as
\[
B_3=\langle \bm a,\bm b, \bm c\ |\ \bm a\cdot\bm b=\bm b\cdot \bm c=\bm c\cdot \bm a\rangle,
\]
after setting $\bm{t}_{(12)}=\bm a$, $\bm{t}_{(23)}=\bm b$, $\bm{t}_{(13)}=\bm c$.
It is easy to see now that this presentation is equivalent to the Artin presentation 
\[
B_3=\langle \bm s_1,\bm s_2\ |\ \bm s_1\cdot\bm s_2\cdot\bm s_1=\bm s_2\cdot\bm s_1\cdot\bm s_2\rangle.
\]
\end{example}

\subsection{The concordance between the two presentations after Bessis and Thm.~\ref{Thm: main}}
\label{Sec: concordance of presn's}

As we said earlier, even though the dual braid presentation of the generalized braid group $B(W)$ has been widely used and with much success, any proofs for it have so far relied on case by case arguments. The original proof by Bessis relied on a statement \cite[Fact~2.2.4]{bessis_dual} about lifts of reflections in the Artin-Tits group, which was proven by computer for the exceptional types and via combinatorial models for the infinite families. Brady and Watt \cite{brady_watt} followed a similar approach. 

Bessis later \cite[Remark~8.9]{Bessis} gave a different geometric proof of Thm.~\ref{Thm: Bessis dual presentation} that relies on Thm.~\ref{Thm: main} and expands on his previous work \cite{Bessis-Zariski}; it proceeds as follows (we try to keep the notation of \cite[\S~4]{Bessis} as much as possible). The Shephard-Todd-Chevalley theorem identifies the space of orbits $V/W$ with an affine complex space $\mathbb{C}^n$, and the free orbits $V^{\op{reg}}/W$ with the complement $\mathbb{C}^n\setminus \mathcal{H}$ of the \emph{discriminant hypersurface} $\mathcal{H}$ of $W$. Then, the generalized braid group $B(W)=\pi_1(\mathbb{C}^n\setminus \mathcal{H})$ can be computed by a version of the classical Zariski--Van-Kampen method. 

In our setting, the hypersurface $\mathcal{H}$ can be realized as a branched covering of degree $n$ over a base space $Y$ and with branch locus $\mathcal{K}\subset Y$. The generic fiber of the projection $p:\mathcal{H}\twoheadrightarrow Y$ has $n$ points, and loops around these points in $\mathbb{C}^n\setminus\mathcal{H}$ generate the fundamental group $B(W)=\pi_1(\mathbb{C}^n\setminus \mathcal{H})$. If we write $f_1,\ldots,f_n$ for these formal generators, the set of relations they must satisfy is determined by the fundamental group $\pi_1(Y\setminus\mathcal{K})$. This is a construction known as the braid monodromy \cite{monodr_2,monodr_3,monodr_1} of the hypersurface $\mathcal{H}$ (with respect to $Y$ and $\mathcal{K}$); we explain further below.

Every element $g\in \pi_1(Y\setminus\mathcal{K})$ acts on a generic fiber of $p$; keeping track of the resulting permutation of the $n$ points gives us a map $\pi_1(Y\setminus\mathcal{K})\rightarrow S_n$ to the symmetric group, known as the \emph{monodromy} of the branched covering; if we moreover keep track of \emph{how} the $n$ points move around each other, we obtain the \emph{braid monodromy}: a map $\Phi:\pi_1(Y\setminus\mathcal{K})\rightarrow B_n$ to the braid group $B_n$. These braids $\phi_g:=\Phi(g)\in B_n$ act on the formal generators $f_i$ with the usual action of $B_n$ on the free group $F_n:=\langle f_1,\ldots,f_n\rangle$ via automorphisms. The  Zariski--Van-Kampen presentation is then given by
\begin{equation}
B(W)=\pi_1(\mathbb{C}^n\setminus\mathcal{H})=\big\langle f_1,\ldots,f_d\ |\ f_i=\phi_g(f_i),\ 1\leq i\leq n,\ g\in\pi_1(Y\setminus\mathcal{K}) \big\rangle.\label{ZVK pres}
\end{equation}

Of course one can simplify the presentation \eqref{ZVK pres} by picking suitable generators of the group $\pi_1(Y\setminus\mathcal{K})$, but this might be a difficult task. Bessis sidesteps this process (and this is a remarkable novelty of his approach) and shows that the whole map $\Phi:\pi_1(Y\setminus\mathcal{K})\rightarrow B_n$ is induced by a \emph{covering map} $LL:Y\setminus\mathcal{K}\rightarrow \op{Conf}_n(\mathbb{C})$, where $\op{Conf}_n(\mathbb{C})$ is the configuration space of $n$ distinct points in the plane. The covering map $LL$ is known as the Lyashko-Looijenga morphism; it corresponds to a subgroup of index $h^nn!/|W|$ in $B_n$, which Bessis identifies as the stabilizer of some fixed reduced reflection factorization of the Coxeter element under the Hurwitz action of $B_n$. Because of this, the Zariski--Van-Kampen presentation \eqref{ZVK pres} can be immediately rewritten as the one in Thm.~\ref{Thm: Bessis dual presentation}. In fact, this is precisely the reason we chose to give the dual braid presentation in that format as opposed to the original version of Bessis.

\subsubsection*{The reliance on the numerological coincidence}\ \newline
The identification of $\Phi\big(\pi_1(Y\setminus\mathcal{K})\big)$ as a stabilizer relies on the fact that the \emph{degree} $h^nn!/|W|$ of the $LL$ map agrees with the number $MC(W)$ of \eqref{Eq: MC(W) as Red_R(c)}; \emph{this is the only ingredient of Bessis' approach that lacked a uniform proof}\footnote{In proving the $K(\pi,1)$ conjecture Bessis also relies on other things, like the lattice property of $NC(W)$, but none of those is used for the dual braid presentation of Thm.~\ref{Thm: Bessis dual presentation}}. Bessis shows that any braid $\phi_g=\Phi(g)\in B_n$ must fix the given reduced factorization, but can only guarantee that all stabilizing braids are realizable as elements $\phi_g\in B_n$ because otherwise the index of $\Phi\big(\pi_1(Y\setminus\mathcal{K})\big)$ in $B_n$ would be higher (see \cite[Thm.~7.5,~Prop.~7.6]{Bessis} and the adjacent discussions). We record the analysis of this section with the following remark.
\begin{remark}
Our uniform argument for Thm.~\ref{Thm: main}, using the $W$-Laplacian and the parabolic recursion (Thm.~\ref{Thm: rec cox nums}) on Coxeter numbers, completes after \cite{Bessis} a case-free proof of the agreement between the dual-braid (\S~\ref{Sec: Dual-braid presentation}) and standard (\S~\ref{Sec: Artin-presentation}) presentations for spherical Artin groups. As we said, the representation-theoretic proof of \cite{Douvr} could also be used for this purpose.
\end{remark}

\begin{remark}
Of course one might further want to completely avoid the numerological coincidence and give an a priori explanation of why the $LL$ map \emph{must} count reduced reflection factorizations of the Coxeter element. The most promising approach towards this is via the theory of Frobenius manifolds as in the work of Hertling and Roucairol, where an equivalent statement \cite[Thm.~7.1]{hertling} is proven for the simply laced finite Coxeter groups in the context of ADE singularities.
\end{remark}

\section{Acknowledgments}

The second author would like to thank Christian Stump, Georges Neaime, and Nathan Williams, for their comments on this work and for many interesting Coxeter-theoretic discussions.

The second author would also like to express his appreciation for the bar Hydronetta, in the greek island Hydra\footnote{Hydra lies across from Spetses, another Greek island famous to representation theorists as the birthplace of the eponymous project \cite{spetses}.}. While attending the online FPSAC 2020 from that establishment, we enjoyed an uninterrupted supply of excellent strawberry daiquiris, consumed primarily as an homage to Henry Miller who wrote some of his best work there \cite{collosus}, which proved crucial to the development of the last two sections of this paper.

\bibliographystyle{alpha}
\bibliography{biblio}

\end{document}